\documentclass{article}
\usepackage[utf8]{inputenc}

\title{A remark on the contactomorphism group \\
       of overtwisted contact spheres}
   
\author{Eduardo Fernández, Fabio Gironella}
\date{}

\usepackage{amsmath,amssymb,amsthm}
\usepackage{mathtools}
\usepackage[margin=120pt]{geometry}
\usepackage[colorlinks=true]{hyperref}
\usepackage{enumitem}
\usepackage{tikz}
  \usetikzlibrary{cd}
  \usetikzlibrary{matrix,arrows}
\usepackage{cleveref}

\usepackage[square,comma]{natbib}
\usepackage{etoolbox}
\makeatletter
    \patchcmd{\NAT@citex}
     {\@citea\NAT@hyper@{%
         \NAT@nmfmt{\NAT@nm}%
         \hyper@natlinkbreak{\NAT@aysep\NAT@spacechar}{\@citeb\@extra@b@citeb}%
         \NAT@date}}
     {\@citea\NAT@nmfmt{\NAT@nm}%
         \NAT@aysep\NAT@spacechar\NAT@hyper@{\NAT@date}}{}{}
    \patchcmd{\NAT@citex}
      {\@citea\NAT@hyper@{%
         \NAT@nmfmt{\NAT@nm}%
         \hyper@natlinkbreak{\NAT@spacechar\NAT@@open\if*#1*\else#1\NAT@spacechar\fi}%
           {\@citeb\@extra@b@citeb}%
         \NAT@date}}
      {\@citea\NAT@nmfmt{\NAT@nm}%
       \NAT@spacechar\NAT@@open\if*#1*\else#1\NAT@spacechar\fi\NAT@hyper@{\NAT@date}}
      {}{}
\makeatother

\theoremstyle{plain}
\newtheorem{thm}{Theorem}
\newtheorem{prop}[thm]{Proposition}
\newtheorem{lemma}[thm]{Lemma}
\newtheorem{cor}[thm]{Corollary}

\theoremstyle{definition}
\newtheorem{definition}{Definition}
\newtheorem{rmk}[thm]{Remark}

\newcommand{\RR}{\mathbb{R}}
\newcommand{\NN}{\mathbb{N}}
\newcommand{\DD}{\mathbb{D}}
\newcommand{\ZZ}{\mathbb{Z}}
\newcommand{\QQ}{\mathbb{Q}}

\newcommand{\SO}{\operatorname{SO}}
\newcommand{\U}{\operatorname{U}}

\newcommand{\SSS}{\mathbb{S}}
\newcommand{\xiot}{\xi_{ot}}
\newcommand{\homgp}[2]{\pi_{#1}\left(#2\right)}
\DeclareMathOperator{\Id}{Id}

\DeclareMathOperator{\Diff}{Diff}
\newcommand{\diff}[1]{\Diff_0\left(#1\right)}
\newcommand{\diffS}{\diff{\SSS^{2n+1}}}
\newcommand{\diffM}{\diff{M}}
\newcommand{\diffSxiot}{\diff{\SSS^{2n+1},\xiot}}

\newcommand{\diffSxiotEight}{\diff{\SSS^{8n+7},\xiot}}
\newcommand{\diffMxi}{\diff{M,\xi}}

\DeclareMathOperator{\Fr}{Fr}

\DeclareMathOperator{\Cont}{Cont}
\DeclareMathOperator{\AlmCont}{AlmCont}
\DeclareMathOperator{\Complex}{Complex}
\DeclareMathOperator{\OT}{OT}
\DeclareMathOperator{\Tight}{Tight}
\newcommand{\Tightk}{\Tight_k}

\newcommand{\OTk}{\OT_k}
\newcommand{\ot}[2]{\OT_{#1}\left(#2\right)}
\newcommand{\complex}[1]{\Complex\left(#1\right)}
\newcommand{\cont}[1]{\Cont\left(#1\right)}
\newcommand{\contOT}[1]{\Cont_{\OT}\left(#1\right)}
\newcommand{\Almcont}[1]{\AlmCont\left(#1\right)}
\newcommand{\almcont}[1]{\AlmCont_{0}\left(#1\right)}

\newcommand{\contMxi}{\cont{M,\xi}}

\DeclareMathOperator{\ev}{ev}
\newcommand{\evN}{\ev_N}

\DeclareMathOperator{\map}{Map}

\DeclareMathOperator{\PLEmb}{Emb_{PL}}
\newcommand{\PLemb}[1]{\PLEmb\left(#1\right)}

\newcommand{\bigslant}[2]{{\raisebox{.2em}{$#1$}\left/\raisebox{-.2em}{$#2$}\right.}}
\newcommand{\quotgps}[2]{\bigslant{\SO\left(#1\right)}{\U\left(#2\right)}}

\def\co{\colon\thinspace}
\def\coeq{\coloneqq\thinspace}

\newcommand{\Addresses}{{
  \bigskip
  \footnotesize
  E.~Fernández, \textsc{Universidad Complutense de Madrid, Departamento de \'Algebra, Geometr\'ia y Topolog\'ia, Facultad de Matem\'{a}ticas, and Instituto de Ciencias Matem\'{a}ticas CSIC-UAM-UC3M-UCM, Madrid, Spain}\par\nopagebreak
  \texttt{eduarf01@ucm.es}\par\nopagebreak
  \smallskip
  
  F.~Gironella, \textsc{Alfréd Rényi Institute of Mathematics, Budapest, Hungary}\par\nopagebreak
  \texttt{fabio.gironella@renyi.hu, fabio.gironella.math@gmail.com}\par\nopagebreak
  \smallskip
  
}}

\begin{document}

\maketitle

\begin{abstract}
    We show the existence of elements of infinite order in some homotopy groups of the contactomorphism group of overtwisted spheres. 
    It follows in particular that the contactomorphism group of some high dimensional overtwisted spheres is not homotopically equivalent to a finite dimensional Lie group.
\end{abstract}

%%%%%%%%%%%%%%%%%%%%%%%%%%%%%%%%%%%%%%%%%%%%%%%%%%%%%%%%%%%%%%%%%%%%%%%%%%%%%%%%%%%%%%%%%%%%%%%%%%%%%
%%%%%%%%%%%%%%%%%%%%%%%%%%%%%%%%%%%%%%%%%%%%%%%%%%%%%%%%%%%%%%%%%%%%%%%%%%%%%%%%%%%%%%%%%%%%%%%%%%%%%

\section{Introduction and statements of the results}
Let $(M,\xi)$ be a closed contact manifold.
These short notes are concerned with the relationship between the topology of the connected component $\diffM$ of the identity in the group of diffeomorphisms of $M$ and its subgroup $\diffMxi$ consisting of contactomorphisms of $(M,\xi)$. 
More precisely, throught the notes we will always assume contact structures to be cooriented and contactomorphisms to be coorientation--preserving.
%Denote by $\contMxi$ the space of contact structures on $M$ which are isotopic to $\xi$. 
%Beside being of independent interest, the topological study of $\diffMxi$ is intimately related with that of the space $\contMxi$ of contact structures and that of the group of diffeomorphisms $\diffM$ (see Equation \ref{eq:ExactSequence}). 

The path components of the group of contactomorphisms of particular contact manifolds have been studied by several authors in the literature; see for instance \cite{Dymara,Giroux,DG,LZ,MN,MasGir15,Vogel,Gironella1,Gironella2}. 
Higher--order homotopy groups have also been studied: for instance, \cite{Eliashberg20years,CP,CS} contain results for the case of the standard tight $(2n+1)$--contact sphere.
In this notes, we deal with the case of overtwisted spheres (cf.\ \cite*{BEM}).

Let $(\SSS^{2n+1},\xiot)$ be any overtwisted sphere, and consider the natural inclusion 
\begin{equation*}
    i\co\diffSxiot\hookrightarrow\diffS \text{ .}
\end{equation*}
For any $k\in\NN$, denote $\mathcal{K}^{2n+1}_k$ the kernel of the homomorphism
\begin{equation*}
\homgp{k}{i}\co\homgp{k}{\diffSxiot}\rightarrow\homgp{k}{\diffS} \text{ .}
\end{equation*}
%\textcolor{red}{(Maybe a notation like $\mathcal{K}^{2n+1}_k$ could do the job too?)}

\begin{thm}\label{thm:StableHomotopy}
	Let $k\in\NN$ be such that $1\leq 4k+1 \leq 2n-1$. 
	The group $\mathcal{K}^{2n+1}_{4k+1}$ contains an infinite cyclic subgroup. 
	%In particular, the fundamental group of $\diffSxiot$ contains an infinite cyclic subgroup.
	%\textcolor{red}{(Why did we specify what the result says in the case of the fund group in the second sentence?)}
\end{thm}

Under some conditions on the dimension, \Cref{thm:StableHomotopy} can be improved in the case of the fundamental group and the fifth homotopy group as follows:
\begin{thm}\label{thm:FundamentalGroup}
	\begin{enumerate}[label=(\roman*)]
	\item \label{item:1_thm2} The group $\mathcal{K}^{3}_1$ contains a subgroup isomorphic to $\ZZ\oplus\ZZ_2$.
	\item \label{item:2_thm2} Let $n\geq3$. The group $\mathcal{K}^{4n+1}_1$ contains a subgroup isomorphic to $\ZZ\oplus\ZZ$.
	\item \label{item:3_thm2} Let $n\geq 6$. The group $\mathcal{K}^{4n+1}_5$ contains a subgroup isomorphic to $\ZZ\oplus\ZZ$.
	\end{enumerate}
\end{thm}

From the methods developed in the paper we are also able to show the following

\begin{thm}\label{thm:pi3}
	\begin{itemize}
		\item [(i)] Let $n\geq 4$. The group
		$\mathcal{K}^{4n+3}_3$ contains an infinite cyclic subgroup.
		\item [(ii)]  Let $n\geq 2$. The group $\mathcal{K}^{8n+7}_4$ contains an infinite cyclic subgroup.
	\end{itemize}
\end{thm}

As the even--order higher homotopy groups of a finite dimensional Lie group are finite (see for instance \citet*[Example 2.51]{FOT}), \Cref{thm:pi3} immediately implies:
\begin{cor}
    \label{cor:not_lie_group}
    For $n\geq2$, $\diffSxiotEight$ is not homotopy equivalent to a finite dimensional Lie group.
\end{cor}

The proofs of Theorems \ref{thm:StableHomotopy}, \ref{thm:FundamentalGroup} and \ref{thm:pi3} use four main ingredients.
The first is the notion of \emph{overtwisted group} introduced in \cite*{CPP}, which relies on the flexibility results for overtwisted contact manifolds from \cite*{EliashbergOT,BEM}. 
The second is the existence of a long exact sequence relating the homotopy groups of the space of contact structures on $\SSS^{2n+1}$ to those of $\diffSxiot$ and of $\diffS$; see \Cref{sec:long_exact_seq}.
The last ingredients are the description of the rational homotopy groups of $\diff{\SSS^{2n+1}}$ from \cite{FH} and the description of some homotopy groups of the homogeneous space $\Gamma_{n}=\SO(2n)/U(n)$ from \cite{Bott,Massey,Harris,Kachi,Mukai}.

%More precisely, we obtain some information on the space of almost contact structures on $\SSS^{2n+1}$ from the results on the homotopy groups of $\Gamma_{n}$ (see \Cref{prop:HomotopyAlmContSpheres}).
%Using the properties of the overtwisted groups, we then prove that some homotopy groups of the space of contact structures contain certain subgroups, namely cyclic or free abelian group on two generators.
%Then, using the knowledge of the rational homotopy groups of $\diff{\SSS^{2n+1}}$, we prove that these subgroups do not come from the homotopy groups of the space of diffeomorphisms in the long exact sequence in \Cref{eq:ExactSequence}.
%In particular, they give the desired subgroups in the kernel of the maps $\pi_k(i)$'s.

We point out that these methods could also be applied to the case of any overtwisted contact manifold $(M^{2n+1},\xi)$ such that both the homotopy type of the space of almost contact structures on $M$ and the diffeomorphism group of $M$ can be (at least partially) understood.

\paragraph*{Acknowledgments}
The authors are extremely grateful to Fran Presas for explaining them the construction of the overtwisted  group and encouraging them to write down this note, as well as to Javier Martínez Aguinaga for very interesting discussions on the problem.
The first author is supported by the Spanish Research Projects SEV--2015--0554, MTM2016--79400--P, and MTM2015--72876--EXP as well as by a Beca de Personal
Investigador en Formaci\'on UCM.
The second author is supported by the grant NKFIH KKP 126683.

%%%%%%%%%%%%%%%%%%%%%%%%%%%%%%%%%%%%%%%%%%%%%%%%%%%%%%%%%%%%%%%%%%%%%%%%%%%%%%%%%%%%%%%%%%%%%%%%%%%%%
%%%%%%%%%%%%%%%%%%%%%%%%%%%%%%%%%%%%%%%%%%%%%%%%%%%%%%%%%%%%%%%%%%%%%%%%%%%%%%%%%%%%%%%%%%%%%%%%%%%%%

\section{Preliminaries}

%%%%%%%%%%%%%%%%%%%%%%%%%%%%%%%%%%%%%%%%%%%%%%%%%%%%%%%%%%%%%%%%%%%%%%%%%%%%%%%%%%%%%%%%%%%%%%%%%%%%%

\subsection{A long exact sequence of homotopy groups}
\label{sec:long_exact_seq}
%\textcolor{blue}{(this might go in the intro actually, because anyway we need the notations at least...)}
%Given a closed smooth manifold $M^{2n+1}$, we denote by $\diffM$ the connected component of the identity in the group of diffeomorphisms of $M$.
%Suppose now that $\xi$ is a contact structure on $M$; we then denote by $\contMxi$ the space of contact structures on $M$ which are isotopic to $\xi$, and by $\diffMxi$ the subgroup of contactomorphisms in $\diffM$.
Let $(M,\xi)$ be a closed contact manifold. 
In this section, the spaces $\diffM$ and $\diffMxi$ are be considered as pointed spaces, with base point $\Id$.
Similarly, $\contMxi$ is considered with base point $\xi$.

As shown for instance in \cite{MasGir15} (and, more in detail, in \cite{MasFibrNotes}), the natural map
\begin{center}
	$\begin{array}{rccl}
	\diffM& \longrightarrow &  \contMxi\\
     \varphi& \longmapsto & \varphi_* \xi
	\end{array}$
\end{center}
is a locally--trivial fibration with fiber $\diffMxi$; 
see also \cite{GeiGon04} for a proof of the fact that the map is a Serre fibration (which is enough for what follows).
In particular, it induces a long exact sequence of homotopy groups
\begin{equation}\label{eq:ExactSequence}
    \ldots \rightarrow \homgp{k+1}{\contMxi}\rightarrow \homgp{k}{\diffMxi} \rightarrow \homgp{k}{\diffM}\rightarrow \homgp{k}{\contMxi}\rightarrow\ldots
\end{equation}

%%%%%%%%%%%%%%%%%%%%%%%%%%%%%%%%%%%%%%%%%%%%%%%%%%%%%%%%%%%%%%%%%%%%%%%%%%%%%%%%%%%%%%%%%%%%%%%%%%%%%

\subsection{Almost contact structures on $\SSS^{2n+1}$}

Recall that, given an oriented smooth manifold $M^{2n+1}$, an \em almost contact structure \em is a triple $(\xi,J,R)$, where $\xi\subseteq TM$ is a cooriented hyperplane distribution, $J:\xi\rightarrow \xi$ is a complex structure on $\xi$, $R=\langle v \rangle \subseteq TM$ is a trivial line sub--bundle defining the coorientation of $\xi$ and $\xi\oplus R\cong TM$ as oriented vector bundles. 
Fixing an auxiliary Riemannian metric $g$ on $M$ which is adapted to $J$ and such that $w$ is of norm $1$ and orthogonal to $\xi$, one can see that $(\xi,J,R)$ is equivalent to a reduction of the structure group $\SO(2n+1)$ of the principal bundle $\Fr_{\SO}(M)$ of orthonormal oriented frames of $TM$ to its subgroup $\U(n)=\U(n)\times 1\subseteq \SO(2n+1)$. 
The space of such reductions is the space of sections $\Gamma(M;X)$ of a fiber bundle $\pi \co X=\Fr_{\SO}(M)/\U(n) \rightarrow M$, with typical fiber $\SO(2n+1)/\U(n)$. 
Such space $\Gamma(M;X)$ is naturally identified with the space of almost contact structures on $M$, which we denote $\AlmCont(M)$.

Recall also (see \citet*[Lemma 8.2.1]{GeiBook}) that there is an identification \begin{equation}\label{eq:GammaGroup}
\Gamma_{n+1}\coeq \quotgps{2n+2}{n+1}\simeq\quotgps{2n+1}{n} \text{ .}
\end{equation} 
In particular, the fiber bundle $\pi$ can also be seen as a fibration
\begin{equation}\label{eq:AlmContFib}
	\begin{tikzcd}
		\Gamma_{n+1}  \ar[r, hook] &
		X \ar[d, "\pi"] \\
		& 
		M
	\end{tikzcd}
\end{equation}

Denote the trivial real line bundle over $M$ by $\varepsilon=\langle w \rangle $.
Then, the Riemannian metric $g$ on $M$ naturally extends to a metric on $TM\oplus\varepsilon$, still denoted $g$, by declaring the vector $w$ to be orthogonal to $TM$ and of norm $1$. 
Let now $\complex{TM\oplus\varepsilon}$ be the space of complex structures on the oriented bundle $TM\oplus\varepsilon$, which are compatible with the metric $g$ (i.e.\ $g(J.,J.)=g(.,.)$). 
Notice that this space can be identified with the space of sections of a fiber bundle over $M$ with fiber the space of complex structures on $\RR^{2n+2}$ compatible with the standard metric, i.e.\ $\Gamma_{n+1}$. 

Given any almost contact structure $(\xi,J,R)$, one can naturally extend $J$ to a complex structure  $\tilde{J}:TM\oplus\varepsilon\rightarrow TM\oplus\varepsilon$ on $TM\oplus\varepsilon$, by defining $\tilde{J}v=-w$. 
This gives an inclusion $j\co\AlmCont(M)\hookrightarrow\Complex(TM\oplus\varepsilon)$. 

In fact, \Cref{eq:GammaGroup} says that $i$ is a diffeomorphism. 
More precisely, denoting the projection on the first factor by $pr:TM\oplus \varepsilon\rightarrow TM$, the map 	
\begin{equation*}
		\begin{array}{rccl}
		\Phi\colon & \Complex(TM\oplus\varepsilon^1) & \longrightarrow &  \AlmCont(M) \\
		& J& \longmapsto & (TM\cap J (TM), J\vert_{TM\cap J (TM)},\langle pr(Jw)\rangle)
		\end{array}
\end{equation*}
is the inverse of $i$.
As a consequence: 
\begin{lemma}\label{lem:AlmostContactFibrationTrivial}
If the vector bundle $TM$ is stably trivial of type $1$ over $\RR$, i.e. $TM\oplus\varepsilon$ is trivializable (as real vector bundle), the fiber bundle $\pi\co X \to M$ is trivializable.
\end{lemma}

%In other words, almost contact structures are just sections of a fiber bundle $\pi \co X \rightarrow M$, with typical fiber $\quotgps{2n+1}{n}$.

%\textcolor{blue}{I think is not necessary to specify the construction. Are you agree?}
%\textcolor{red}{to specify here the definition of $X$ and $\pi$; with transition functions it should be short to state. Also, is $X$ a good notation? I don't know what's the notation in the literature, we should check.}
%\\
%Recall also that $\quotgps{2n+1}{n}$ is homotopy equivalent to $\Gamma_{n+1}\coeq \quotgps{2n+2}{n+1}$; in particular, the fiber bundle $\pi$ can also be seen as a fibration
%\begin{center}
%\begin{tikzcd}
  %\Gamma_{n+1}  \ar[r, hook] &
  %X \ar[d, "\pi"] \\
  %& 
  %M
%\end{tikzcd}
%\end{center}

For the rest of the section we focus on the case of almost contact structures on $\SSS^{2n+1}$.

According to \Cref{lem:AlmostContactFibrationTrivial}, the fiber bundle $\pi \co X\rightarrow \SSS^{2n+1}$ is trivial. 
Once fixed any trivialization, one can then identify $\AlmCont(\SSS^{2n+1})=\map(\SSS^{2n+1},\Gamma_{n+1})$. 

\begin{rmk}
\label{rmk:HomotopyGamma}
	The homotopy groups $\pi_k (\Gamma_{n+1})$, in the stable range $1\leq k\leq 2n$, were computed in \cite{Bott}: they are of period $8$ and the first eight groups are, in order, $0,\ZZ,0,0,0,\ZZ,\ZZ_2,\ZZ_2$.
	Moreover, some of the first unstable groups $\homgp{2n+1+k}{\Gamma_{n+1}}$ were computed in \cite{Massey,Harris,Kachi,Mukai}. 
	More precisely, we will use the fact that the following unstable homotopy groups contain a cyclic subgroup:
	$\homgp{4n+3}{\Gamma_{2n+1}}$, $\homgp{4n+7}{\Gamma_{2n+1}}$, 	$\homgp{4n+7}{\Gamma_{2n+2}}$ and $\homgp{8n+12}{\Gamma_{4n+4}}$.
\end{rmk}

%In what follows, we denote the north pole of $\SSS^{2n+1}$ by $N$.

\begin{lemma}\label{lemma:PathComponentsHomeo}
	All the path connected components of the space $\AlmCont(\SSS^{2n+1})$ are homeomorphic.
\end{lemma}
\begin{proof}
    Let $J_0\in\Gamma_{n+1}$ be the standard (almost) complex structure on $\RR^{2n+2}$, and denote
    \begin{align*}
        \xi_0\co \SSS^{2n+1} & \to\Gamma_{n+1} \\
        z & \mapsto J_0
    \end{align*}
    the almost contact structure with constant value $J_0$. 
    Consider then any other almost contact structure $\xi:\SSS^{2n+1}\rightarrow\Gamma_{n+1}$.
    Because $\Gamma_{n+1}$ is path--connected, up to homotopy, we can moreover assume that $\xi(N)=J_0$, where $N$ denotes the north pole of $\SSS^{2n+1}$. 
    
    Denote by $\AlmCont_{\xi_0}(\SSS^{2n+1})$ and $\AlmCont_\xi (\SSS^{2n+1})$ the path connected components of $\xi_0$ and $\xi$, respectively. 
    Consider the $\U(n+1)$--principal bundle  $p:\SO(2n+2)\rightarrow \Gamma_{n+1}$, $A\mapsto A\cdot J_0 \cdot A^{-1}$.
    By Bott periodicity, $\pi_{2n}(\U(n+1))=0$.
    In particular, the homomorphism
    \begin{equation*}
        \homgp{2n+1}{p}\co\homgp{2n+1}{\SO(2n+2)}\to \homgp{2n+1}{\Gamma_{n+1}}
    \end{equation*}
    is surjective, so that there exists a lift $\widehat{\xi}:\SSS^{2n+1}\rightarrow\SO(2n+2)$ of $\xi$ such that $\widehat{\xi}(N)=\Id$. 
    
    The desired homeomorphism is then given by
    \begin{equation*}
	    \begin{array}{rccc}
	    \Phi_{\hat{\xi}}\colon & \AlmCont_{\xi_0}(\SSS^{2n+1}) & \longrightarrow &  \AlmCont_{\xi}(\SSS^{2n+1}) \\
	    & \eta& \longmapsto & \widehat{\xi}\cdot\eta
	    \end{array}
    \end{equation*}
    where 
    \begin{align*}
        \widehat{\xi}\cdot\eta\co \SSS^{2n+1} &\to\Gamma_{n+1}\\ 
        z & \mapsto\hat{\xi}(z)\cdot\eta(z)
    \end{align*}
    is defined by using the left action of $\SO(2n+2)$ on $\Gamma_{n+1}$.
\end{proof}

\begin{prop}\label{prop:HomotopyAlmContSpheres}
		For each $k\in\NN$ there is an isomorphism  \begin{equation*}
		    \homgp{k}{\AlmCont(\SSS^{2n+1})}\cong\homgp{k}{\Gamma_{n+1}}\oplus\homgp{2n+k+1}{\Gamma_{n+1}}
		\end{equation*}
\end{prop}
\begin{proof}
	For $k=0$ we argue as follows.
	Recall that $[\SSS^n,X]=\homgp{n}{X,x}/\homgp{1}{X,x}$, for any pointed topological space $(X,x)$. 
	Hence, 
	\begin{equation*}
	\homgp{0}{\AlmCont(\SSS^{2n+1})}=[\SSS^{2n+1},\Gamma_{n+1}]=\homgp{2n+1}{\Gamma_{n+1}}/\homgp{1}{\Gamma_{n+1}}=\homgp{2n+1}{\Gamma_{n+1}} \text{ ,}
	\end{equation*}
	and the statement follows from the fact that, according to \Cref{rmk:HomotopyGamma}, $\Gamma_{n+1}=\SO(2n+2)/\U(n+1)$ is simply connected.
	
	We now prove the statement for $\pi_k$ with $k\geq 1$.
	According to \Cref{lemma:PathComponentsHomeo}, we can consider $\almcont{\SSS^{2n+1}}$ as space pointed at $\xi_0\equiv J_0:\SSS^{2n+1}\rightarrow\Gamma_{n+1}$.
	Similarly, we consider $\Gamma_{n+1}$ as space pointed at $J_0$.
	There is then a natural Serre fibration (of pointed spaces)
	\begin{align*}
	    \evN\co \AlmCont(\SSS^{2n+1}) &\to \Gamma_{n+1}\\
	    \xi\quad\quad &\mapsto\xi(N)
	\end{align*}
	\newline
	The fiber over $J_0$ is the space $F\coeq \AlmCont_{\xi(N)=J_0}(\SSS^{2n+1})$ of almost contact structures which evaluate at $J_0$ on the north pole,
	which is naturally considered as pointed at $\xi_0$.
	In particular, $F=\map((\SSS^{2n+1},N),(\Gamma_{n+1},J_0))$, so that $\pi_k(F)=\pi_{2n+k+1}(\Gamma_{2n+1})$.

	Moreover, the map
	\begin{align*}
	    s\co\Gamma_{n+1}&\rightarrow\AlmCont_{\xi_0}(\SSS^{2n+1})\\
	    J&\mapsto \quad \xi_J
	\end{align*}
	where $\xi_J\equiv J$, defines a section of the fibration. 
	In particular, the boundary map in the long exact sequence of homotopy groups associated to the Serre fibration $ev_N$ is trivial, and every obtained short exact sequence of groups splits.
	In other words, 
	\begin{equation*}
		    \homgp{k}{\AlmCont(\SSS^{2n+1})}\cong\homgp{k}{\Gamma_{n+1}}\oplus\homgp{k}{F} =\homgp{k}{\Gamma_{n+1}}\oplus \pi_{2n+k+1}(\Gamma_{2n+1})\text{ .}
		    \qedhere
	\end{equation*}
\end{proof}

\subsection{The overtwisted group}
\label{sec:OT_gp}

Let $M$ be a $(2n+1)$--dimensional manifold.
We denote in this section the subspaces of contact and almost contact structures on $M$ with a fixed overtwisted disk $\Delta_0\subset M$ respectively by $\contOT{M,\Delta_0}\subseteq\cont{M}$ and $\Almcont{M,\Delta_0}\subseteq\Almcont{M}$. 
\begin{thm}[\citet*{EliashbergOT,BEM}]
    \label{thm:OT_str}
    The following forgetful map induces a weak homotopy equivalence:
    \begin{equation*}
        \contOT{M,\Delta_0}\rightarrow\Almcont{M,\Delta_0},
    \end{equation*}
\end{thm}

Notice that the overtwisted disk is not allowed to move in this results.
However, an easy corollary is the fact that the forgetful map
\begin{equation}
    \label{eq:BEMPathComponents}
    \contOT{M}\rightarrow \Almcont{M}
\end{equation} 
induces a bijection at $\pi_0$--level, where $\contOT{M}$ denotes the space of overtwisted contact structures on $M$. 
This can be seen by introducing an overtwisted disk in a neighborhood of a (properly chosen) point of $M$, and using \Cref{thm:OT_str}.

To deal with the higher--order homotopy groups, one needs the existence of a continuous choice of overtwisted disks in order to run the same argument. 
%\textcolor{blue}{
%\begin{definition}[\cite{CPP}]
    %\label{def:OTgroup}
%	Let $0\leq k\leq 2n$.
%	The \emph{overtwisted $k$ group} of $M$, denoted by $\OTk(M)$, is the subgroup of homotopy classes of $k$--spheres $(\xi,\Delta)\co\SSS^k\rightarrow\cont{M}$ of overtwisted contact structures which admits a continuous choice of overtwisted disk. A homotopy class in $\OTk(M)$ is called \em overtwisted. \em A class not belonging to $\OTk(M)$ is a called a \em tight class.\em 
%\end{definition}
%}

\begin{definition}[\cite*{CPP}]
    \label{def:OTgroup}
	Let $0\leq k\leq 2n$.
	The \emph{overtwisted $k$--group} of $M$, denoted $\OTk(M)$, is the subgroup of $\pi_k(\contOT{M})$ made of those classes that admit a representative $\xi\co\SSS^k\rightarrow\contOT{M}$ for which there is a \emph{certificate of overtwistedness}, i.e.\ a continuous map 
	\begin{equation*}
	    \Delta \co \SSS^k \to \PLemb{\DD^{2n},M} \coeq \{\,\psi\co \DD^{2n} \hookrightarrow M \text{ piece--wise linear embedding}\,\}
	\end{equation*}
	such that, for each $p\in\SSS^k$, $\Delta(p)$ is overtwisted for $\xi(p)$.
\end{definition}
	Homotopy classes in $\OTk(M)$ are called \emph{overtwisted}. A homotopy class which is not overtwisted is called \emph{tight}.

In these terms, \Cref{eq:BEMPathComponents} says that the map $\ot{0}{M}\rightarrow\homgp{0}{\Almcont{M}}$ is a bijection.
For higher--order homotopy groups one then has the following:
\begin{prop}[{\citet*[Proposition 33]{CPP}}]
    \label{prop:OTGroup}
    Let $(M,\xiot)$ be any closed overtwisted contact manifold. For each $0\leq k\leq 2n$, the inclusion $\contOT{M}\hookrightarrow\Almcont{M}$ induces an isomorphism 
    \begin{equation*}
    \OTk(M)\xrightarrow{\sim}\homgp{k}{\Almcont{M}} .
    \end{equation*}
    Moreover, $\OTk(M)<\homgp{k}{\contOT{M},\xiot}=\homgp{k}{\cont{M},\xiot}$ is a normal subgroup for $k>0$ and, thus, the set of tight classes $\Tightk(M)=\homgp{k}{\cont{M},\xiot}/\OTk(M)$ has group structure. 
    In particular, for any $1\leq k\leq 2n$ there is an isomorphism 
    \begin{equation*}
        \homgp{k}{\cont{M},\xiot}\cong\OTk(M)\oplus\Tightk(M)\text{ .}
    \end{equation*}
\end{prop}
To the authors' knowledge, the only known example of a non--trivial tight class is contained in \cite{Vogel}, where the author exhibits an order $2$ loop of overtwisted contact structures on $\SSS^3$, based at the only overtwisted structure on $\SSS^3$ having Hopf invariant $-1$, which does not admit a certificate of overtwistedness. 
It follows that this tight loop cannot come from a loop of diffeomorphisms in the long exact sequence in \Cref{eq:ExactSequence}. 
In particular, its image via the boundary map is a non--trivial element (of order $2$) in the contact mapping class group.

\section{Proofs of the statements}
\label{sec:proofs}

We start by recalling some known facts in algebraic topology.
Recall the following standard homotopy equivalence (see for instance \citet*[Lemma 1.1.5]{ABK} for a proof):
\begin{equation}
    \label{eq:DiffSpheres}
    \diff{\SSS^{2n+1}}\xleftarrow{\sim}\diff{\DD^{2n+1},\partial}\times\SO(2n+2) \text{ .}
\end{equation}
Here, the group $\diff{\DD^{2n+1},\partial}$ of diffeomorphisms of the disk relative to its boundary which are smoothly isotopic to the identity is understood as the subgroup of $\diff{\SSS^{2n+1}}$ of diffeomorphisms which fixes (a neighborhood of) the north hemisphere, and the arrow is the natural inclusion map. 
Moreover, some of the rational homotopy groups of the first factor of the right--hand side of \Cref{eq:DiffSpheres} are completely characterized (see also \citet*[Section 6]{WW}):
\begin{thm}[\cite{FH}]
    \label{thm:RationalHomotopyDiffDisk}
	Let $0\leq k<\min\{\frac{2n-3}{3},n-3\}$. Then
		\begin{equation*}
            \homgp{k}{\diff{\DD^{2n+1},\partial}}\otimes\QQ=
            \begin{cases}
                0& \text{if $k\not\equiv 3\mod 4,$} \\
	            \QQ& \text{if $k\equiv 3\mod 4$.} 
	        \end{cases}
	\end{equation*}
\end{thm}

Let's now go back to contact topology and prove the statements announced in the introduction.
\begin{proof}[Proof (\Cref{thm:StableHomotopy})] 
    Let $\xiot$ be any overtwisted structure on $\SSS^{2n+1}$, and $k\in\NN$ such that $1\leq 4k+1\leq 2n-1$.
    The relevant part of the long exact sequence in  \Cref{eq:ExactSequence} is the following:
    \begin{equation*}
    \begin{tikzcd}
		\pi_{4k+2}(\diff{\SSS^{2n+1}}) \arrow[r] &
		\pi_{4k+2}(\cont{\SSS^{2n+1}}) \arrow[r] &
		%\pi_{4k+1}(\diff{\SSS^{2n+1},\xiot})_{0}
		\mathcal{K}^{2n+1}_{4k+1}
	\end{tikzcd}
    \end{equation*}
    
    According to Propositions \ref{prop:HomotopyAlmContSpheres} and \ref{prop:OTGroup}, there is an isomorphism 
    \begin{equation*}
        \homgp{4k+2}{\cont{\SSS^{2n+1},\xiot}}\cong\homgp{4k+2}{\Gamma_{n+1}}\oplus\homgp{2n+4k+3}{\Gamma_{n+1}}\oplus\Tightk(\SSS^{2n+1}).
    \end{equation*}
    Moreover, under this isomorphism, the projection on the first factor 
        \begin{equation*}
        \homgp{4k+2}{\cont{\SSS^{2n+1},\xiot}}\to\homgp{4k+2}{\Gamma_{n+1}}\text{ .}
    \end{equation*}
    is just the map induced by the evaluation at the north pole $ev_N$.
    As $\diff{\DD^{2n+1},\partial}\overset{i}{\subset}\diff{\SSS^{2n+1}}$ is the subgroup of diffeomorphisms fixing the north hemisphere, it follows that the following composition is trivial:
    \begin{equation*}
    \begin{tikzcd}
		\pi_{4k+2}(\diff{\DD^{2n+1},\partial}) \arrow[r, "\pi_{4k+2}(i)"] &
		\pi_{4k+2}(\diff{\SSS^{2n+1}}) 
		\arrow[dl, out=0, in=180]
		\\
		\pi_{4k+2}(\cont{\SSS^{2n+1}}) \arrow[r, "\pi_{4k+2}(ev_N)"] &
		\homgp{4k+2}{\Gamma_{n+1}}
	\end{tikzcd}
    \end{equation*}

    Moreover, according to Bott periodicity, $\homgp{4k+2}{\SO(2n+2)}=0$.
    In particular, the following composition is also trivial:
    \begin{equation*}
    \begin{tikzcd}
		\pi_{4k+2}(\diff{\SSS^{2n+1}}) \arrow[r] &
		\pi_{4k+2}(\cont{\SSS^{2n+1}}) \arrow[rr, "\pi_{4k+2}(ev_N)"] &
		&
		\homgp{4k+2}{\Gamma_{n+1}}
	\end{tikzcd}
    \end{equation*}

    Now, according to \Cref{rmk:HomotopyGamma}, $\homgp{4k+2}{\Gamma_{n+1}}$, hence $\homgp{4k+2}{\cont{\SSS^{2n+1}}}$, contains a subgroup $\ZZ$.
    It then follows from the exact sequence that $\mathcal{K}^{2n+1}_{4k+1}$ must have at least one element of infinite order, as desired.
\end{proof}

\begin{proof}[Proof (\Cref{thm:FundamentalGroup})]
According to \cite{HatcherSmale}, the Smale Conjecture holds for $\SSS^3$; in particular, $\homgp{2}{\diff{\SSS^3}}=0$. 
Moreover, since $\Gamma_2=\SO(4)/\U(2)=\SSS^2$ it follows from Propositions \ref{prop:HomotopyAlmContSpheres} and \ref{prop:OTGroup} that the group
\begin{equation*}
    \OT_2(\SSS^3)\cong\homgp{2}{\SSS^2}\oplus\homgp{5}{\SSS^2}\cong\ZZ\oplus\ZZ_2
\end{equation*}    
is a subgroup of $\homgp{2}{\cont{\SSS^3},\xiot}$. 
\Cref{item:1_thm2} then follows from the exact sequence in \Cref{eq:ExactSequence}.

\sloppy{Since $\homgp{2}{\SO(4n+1)}=\homgp{6}{\SO(4n+1)}=0$, \Cref{thm:RationalHomotopyDiffDisk} implies that $\homgp{2}{\diff{\SSS^{4n+1}}}\otimes\QQ=0$ for $n\geq 3$, and $\homgp{6}{\diff{\SSS^{4n+1}}}\otimes\QQ=0$	for $n\geq 6$.}
Moreover, according to \Cref{rmk:HomotopyGamma}, each of the following homotopy groups contain a cyclic subgroup: $\homgp{2}{\Gamma_{2n+1}}$ for $n\geq 1$, $\homgp{6}{\Gamma_{2n+1}}$ for $n\geq 2$, $\homgp{4n+3}{\Gamma_{2n+1}}$ and $\homgp{4n+7}{\Gamma_{2n+1}}$. 
    \Cref{item:2_thm2,item:3_thm2} then follow from the exact sequence in \Cref{eq:ExactSequence} and from Propositions \ref{prop:HomotopyAlmContSpheres} and \ref{prop:OTGroup}.
\end{proof}

\begin{proof}[Proof (\Cref{thm:pi3})]
	Since $\homgp{4}{\SO(4n+4)}$ is trivial, it follows from the identification in \Cref{eq:DiffSpheres} and from \Cref{thm:RationalHomotopyDiffDisk} that $\homgp{4}{\diff{\SSS^{4n+3}}}\otimes\QQ=0$ for $n\geq 4$. 
	Moreover, according to \Cref{rmk:HomotopyGamma}, $\homgp{4n+7}{\Gamma_{2n+2}}$ contains a subgroup $\ZZ$.
	
	Similarly, $\homgp{5}{\SO(8n+8)}=0$ thus \Cref{eq:DiffSpheres} and Theorem \ref{thm:RationalHomotopyDiffDisk} imply that $\homgp{5}{\diff{\SSS^{8n+7}}}\otimes\QQ=0$ for $n\geq 2$. 
	According to \Cref{rmk:HomotopyGamma}, $\homgp{8n+12}{\Gamma_{4n+4}}\cong\ZZ$.
	
	The statement then follow from the exact sequence in \Cref{eq:ExactSequence} and from Propositions \ref{prop:HomotopyAlmContSpheres} and \ref{prop:OTGroup}.
\end{proof}

%%%%%%%%%%%%%%%%%%%%%%%%%%%%%%%%%%%%%%%%%%%%%%%%%%%%%%%%%%%%%%%%%%%%%%%%%%%%%%%%%%%%%%%%%%%%%%%%%%%%%
%%%%%%%%%%%%%%%%%%%%%%%%%%%%%%%%%%%%%%%%%%%%%%%%%%%%%%%%%%%%%%%%%%%%%%%%%%%%%%%%%%%%%%%%%%%%%%%%%%%%%

\bibliographystyle{abbrvnat}
\bibliography{biblio}

\Addresses

\end{document}